\documentclass{amsart}
\usepackage{main} 
\usepackage{parskip}

\title{Hilbert schemes and Seshadri constants}

\author[Jonas Baltes]{Jonas Baltes}
\address{Georg-August-Universität Göttingen, Mathematisches Institut, Bunsenstraße 3-5, 37073 Göttingen, Germany}
\email{jonas.baltes@mathematik.uni-goettingen.de}
\urladdr{https://sites.google.com/view/jonasbaltes}
\date{\today}

\begin{document}
\begin{abstract}
 In this paper we will propose a new method to investigate Seshadri constants, namely by means of (nested) Hilbert schemes. This will allow us to use the geometry of the latter spaces, for example the computations of the nef cone via Bridgeland stability conditions to gain new insights and bounds on Seshadri constants. Moreover, it turns out that many known Seshadri constants turn up in the wall and chamber decomposition of the movable cone of Hilbert schemes.\par
    
\end{abstract}
\maketitle
\numberwithin{theorem}{section}
\section{Introduction}
In \cite{demaillySeshadri} Jean-Pierre Demailly introduced Seshadri constants trying to prove the Fujita conjecture. Since then these constants sparked a lot of interest even without the link to the Fujita conjecture: Examples are the Nagata conjecture with its relation to Hilbert's 14\textsuperscript{th} problem \cite{nagataOriginalPaper}, symplectic packings \cite{biranSymplectic} and several conjectures concerning the Seshadri constants themselves, see e.g.\@ the survey paper \cite{BauerPrimer}. This interest is partially due to the fact that these numbers are in general very hard to come by and they have only been explicitly computed in very few cases.\par

The purpose of this paper is manifold. The overall goal is to show and convince the reader of the usefulness of Hilbert schemes for the study of Seshadri constants. We will proceed as follows:
\begin{itemize}
    \item In \Cref{sec:Observation} we will present numerical findings that propose a direct link between Seshadri constants and walls in the cones of the Hilbert scheme (of length $3$).
    \item In \Cref{sec:ConesOfNested} we will compute all Seshadri constants in terms of cones of the nested Hilbert scheme. In the following \Cref{sec:applications} we link this to the Nagata conjecture and give some applications.
    \item In \Cref{sec:BoundsOnSeshadri} we will give new bounds on Seshadri constants on surfaces by using the computation of the nef cone of the Hilbert scheme.
\end{itemize}
We will now briefly introduce the two main notations. Let $X$ be a smooth projective surface over $\CC$ and $H\in \Pic X$ an ample line bundle.
\begin{itemize}
    \item For any collection of distinct points $x_1,\ldots, x_r\in X$ we define
\begin{equation*}
    \varepsilon(H, x_1,\ldots, x_r) = \inf_{C\subset X} \frac{C.H}{\sum m(C,x_i)},
\end{equation*}
where the infimum is taken over all curves $C\subset X$ and $m(C,x_i)$ is the multiplicity of a curve $C$ at $x_i$. Taking the infimum of this yields the global Seshadri constant $$\varepsilon(H) = \inf_{x\in X} \varepsilon(H,x),$$ and the supremum yields $$\varepsilon(H,r) = \inf_{x_1,\ldots, x_r\in X} \varepsilon(H,x_1, \ldots, x_r).$$ 
    \item The Hilbert scheme of length $n\in \NN$ is defined as
    \begin{equation*}
        X^{[n]} = \{Z\subset X\,|\, \dim Z = 0 \textup{ and }H^0(Z,\oO_Z) = n\}.
    \end{equation*}
    If $h^1(\oO_X) = 0$ we have $\Pic X^{[n]} = \Pic X \oplus B$, where the divisor $2B$ is the divisor of non-reduced schemes.
\end{itemize}
Sticking now to K3 surfaces, we can make use of the results by Bayer--Macrì \cite{bayer2014mmp} who computed the wall-and-chamber decomposition of the movable cone $\Mov X^{[n]}$ with the help of Bridgeland stability conditions.
It turns out that everything is determined by certain Pell-equations, see \Cref{cor:baymacri}. Computing this with the help of the program \cite{python} written by the author, yields the following observation:
\begin{observation}
\label{obs:mainobservation}
    Let $X$ be a K3 surface with $\Pic X = \ZZ H$ of degree\footnote{These are the only degrees in which the Seshadri constants are known.} $H^2 = 2,4,6,8$ or $a^2$ for $a\in 2\NN$. Then $H^{[3]}-\frac{\varepsilon(H)}{2} B$ is a divisor on a wall of $\Mov X^{[3]}$.
\end{observation}

With the help of the nested Hilbert scheme
\begin{equation*}
    X^{[r,r+1]} = \{(Z,Z')\in X^{[r]}\times X^{[r+1]}\,|\, Z\subset Z'\}
\end{equation*}
we will prove the following theorem:
\begin{theorem}[See \Cref{thm:vareps} and \Cref{thm:SeshadriNagataOnNested} for the precise statement]
    Let $X$ be a surface. Then 
    \begin{itemize}
        \item $\varepsilon(H)$ is determined by the nef cone of $X^{[1,2]}$.\footnote{Here we actually prove a stronger statement allowing also infimums over mutliple points which are then determined by the nef cone of $X^{[r,r+1]}$.}
        \item If $\Pic X = \ZZ H$ then $\varepsilon(H,r)$ is determined by $\Eff(X^{[r,r+1]})$
    \end{itemize}
\end{theorem}
The advantage of the approach above to Seshadri constants is, that for many spaces the cones of the usual Hilbert scheme are known, see e.g. \cite{bayer2014mmp} for K3 surfaces or \cite{BologneseNefConeHilbert} for the nef cone on any surface. However, for the nested Hilbert scheme the picture is much more unclear. Ryan--Yang \cite{ryanNefConeNested} recently computed the nef cone of $X^{[r,r+1]}$ for some rational surfaces for all $r>0$ and for very general K3 surfaces if $r\gg 0$.

On the other hand this directly allows to link the long standing Nagata conjecture to nested Hilbert schemes:
\begin{conjecture}[(Biran-)Nagata conjecture]
\label{conj:NagataBiran}
    Let $X$ be a surface and $H\in \Pic X$ an ample divisor. Then
    \begin{equation*}
        \varepsilon(H,r) = \sqrt{\frac{H.H}{r}}
    \end{equation*}
    for all $r\gg 0$.
\end{conjecture}

Using the computations of the nef cone of Hilbert schemes from \cite{BologneseNefConeHilbert}, which was done via Bridgeland stability conditions, we are able to give the following bound on Seshadri constants for non-rational surfaces with $\Pic X = \ZZ H$.
\begin{theorem}
    Let $X$ be a surface with Picard group $\Pic X = \ZZ H$ with effective generator $H$ such that $K_X = aH$ with $a\ge 0$. Then
    \begin{equation*}
        \varepsilon(H) \ge \frac{2H^2n}{(a+1)H^2+n(n+1)+2}
    \end{equation*}
    for all integers $n$ such that $n^2+n \ge H^2$.  
\end{theorem}
For example for very general K3 surfaces this gives bounds that are better than previously known bounds by Knutsen \cite{knutsenGeneralK3} in $50\%$ percent of the cases.
\subsection*{Acknowledgements}
This paper is mostly taken from the fourth chapter of my PhD-thesis \cite{thesis} which was written under the supervision of Frank Gounelas. \par
I would like to express my deep gratitude towards my advisor for the discussions on this topic and helpful comments on draft versions and towards Andreas Knutsen for his discussions during his visit in Göttingen. Moreover, I would like to thank thank Matthias Schütt for several comments on the draft.
\section{Definitions and Numerical Observations}
\label{sec:Observation}
We first start with the following well known definition, see e.g.\@ the survey paper \cite{BauerPrimer}: 
\begin{definition}
    Let $X$ be a surface, $H$ an ample divisor and $Z\subset X$ a $0$-dimensional subscheme of length $r$. Denote by $\Bl_Z X\to X$ the blowup along $Z$ and let $E_r$ be the corresponding exceptional divisor. Then we define
    \begin{equation*}
        \varepsilon(H, Z) = \sup \{\lambda>0\,|\, H-\lambda E_r \in \Pic(\Bl_Z X)_\RR \textup{ is ample}\}.
    \end{equation*}
\end{definition}
For $Z$ the disjoint union of points this recovers the usual definition of multi-point Seshadri constants. In this paper, we allow the infimum and supremum over more points:
\begin{definition}
    Let $X$ be a surface and $H$ an ample divisor. We then define
    \begin{equation*}
        \varepsilon_{\textup{inf}}(H, r) = \inf_{Z\in X^{[r]}} \varepsilon(H,Z)
    \end{equation*}
    and 
    \begin{equation*}
        \varepsilon_{\textup{sup}}(H, r) = \sup_{Z\in X^{[r]}} \varepsilon(H,Z).
    \end{equation*}
\end{definition}
\begin{remark}
    Clearly, $$\varepsilon_{\textup{inf}}(H, 1) = \varepsilon(H)$$ as all closed subschemes of dimension $0$ and length $1$ are closed points. Furthermore, $$\varepsilon_{\textup{sup}}(H, r) = \varepsilon(H,r),$$
    see \Cref{prop:supSeshadri}.
\end{remark}
\subsection*{Examples of Seshadri constants}
In this subsection we will give some known computations and also some conjectures related to Seshadri constants. We will mostly deal with K3 surfaces.
For $X$ a K3 surface over $\CC$ with $\Pic(X) = \ZZ H$ the following numbers are known:
\begin{itemize}
    \item $H.H = 2$ then $\varepsilon(H) = 1$, see e.g.\@ the survey paper \cite{BauerPrimer}.
    \item $H.H = 4$ then $\varepsilon(H) = 2$, see Bauer \cite{bauerQuarticSeshadri}.
    \item $H.H = 6$ then $\varepsilon(H) = 2$, see Galati--Knutsen \cite{knutsenSeshadri6and8}.
    \item $H.H = 8$ then $\varepsilon(H) = \frac{8}{3}$, see Galati--Knutsen \cite{knutsenSeshadri6and8}.
    \item $H.H = a^2$ a perfect square, then $\varepsilon(H) = a$, see Knutsen \cite{knutsenGeneralK3}.
\end{itemize}
Moreover Galati--Knutsen give some evidence in \cite{knutsenSeshadri6and8} that the following should hold as well:
\begin{itemize}
    \item $H.H = 14$ then $\varepsilon(H) = \frac{14}{4}$?
    \item $H.H = 24$ then $\varepsilon(H) = \frac{24}{5}$?
\end{itemize}
For very general Enriques surfaces Galati--Knutsen \cite{galati2022rationalEnriques} proved the following: For $H\in \Pic X$ ample and the Cossec function $$\phi(H) = \min\{E.H\,|\, E^2=0, E \textup{ effective and non-trivial}\}$$
we have
\begin{equation*}
    \varepsilon(H) = \phi(H).
\end{equation*}

Furthermore for arbitrary surfaces there is the following conjecture:
\begin{conjecture}[Farnik--Szemberg--Szpond--Tutaj-Gasinska \cite{szembergConjecture}]
\label{conj:szemberg}
    Let $X$ be a surface with Picard group $\ZZ H$ such that $H^2$ is not a perfect square. Let $(q,p)$ be a primitive solution of $p^2-H^2q^2 = 1$. Then
    \begin{equation*}
        \varepsilon(H,1) \ge \frac{q}{p}H^2,
    \end{equation*}
    where $\varepsilon(H,1) = \max_{x\in X} \varepsilon(H,x)$.
\end{conjecture}

\subsection*{Cones and walls of the Hilbert scheme}
In the following we want to explain the relation of Seshadri constants and Hilbert schemes. Therefore, we explain the wall-and-chamber decomposition of the movable cone of the Hilbert scheme of K3 surfaces, following \cite{bayer2014mmp}. The main result of loc.\@cit.\@ is the computation of the birational behaviour of any moduli space of sheaves on a K3 surface. This then leads to the following wall-and-chamber decomposition for a Hilbert scheme $X^{[n]}$ 
\begin{equation*}
    \Mov(X^{[n]}) = \bigcup_{i\colon X\dashrightarrow \Tilde{X}} i^*\Nef(\Tilde{X}),
\end{equation*}
where the union runs over all non-trivial birational $K$-trivial models $\Tilde{X}$, see \cite[Section 3.7]{debarre2018hyperk}.
To describe this decomposition, we need to introduce the lattice $$\ZZ\oplus \Pic(S)\oplus \ZZ =\colon \Halg(S)$$ with the intersection pairing $$(a,D,b)\cdot (c,D',d) = D.D'-ad-cb.$$ Then, for $n>1$ the Picard group $$\Pic(X^{[n]}) \cong (1,0,-n+1)^\perp \cong \Pic(X)\oplus \ZZ B,$$ where $B = B^{[r]} = (1,0,n-1)$. Via this identification and $v=(1,0,-n+1)$, the following theorem holds.
\begin{theorem}[Bayer--Macrì {\cite[Theorem 12.1, 12.3]{bayer2014mmp}}]
    The movable cone of $X^{[n]}$ is cut out in $\mathrm{Pos}(X^{[n]})\otimes \RR$ by
    \begin{itemize}
        \item $s^\perp\cap v^\perp$, where $s\cdot s = -2$
        \item $w^\perp \cap v^\perp$, where $w\cdot w = 0$ and $w\cdot v \in \{1,2\}$.
    \end{itemize}
    The walls in $\Mov(X^{[n]})$ are given by $v^\perp\cap a^\perp$ (when non-trivial) for all $a\in \Halg(X)$ with $a\cdot a \ge -2$ and $0\le a\cdot v \le v\cdot v/2$.
\end{theorem}
In particular, if $\Pic(X) = \ZZ H$ with $H.H = 2d$ and $H^{[n]} = (0,H,0)$, this yields:
\begin{corollary}[Bayer--Macrì {\cite[Proposition 13.1]{bayer2014mmp}}]
\label{cor:baymacri}
    \begin{itemize}
        \item If $d = \frac{k^2}{h^2}(n-1)$, then $$\Mov(X^{[n]}) = \langle H^{[n]}, H^{[n]}-\frac{k}{h}B\rangle.$$
        \item Otherwise if the equation $(n-1)x^2-dy^2 = 1$ has a solution $(x_1,y_1)$ with $x_1>0$ minimal and $y_1>0$, then $$\Mov(X^{[n]}) = \langle H^{[n]}, H^{[n]}-\frac{dy_1}{x_1(n-1)}B\rangle.$$
        \item Otherwise let $(x_1,y_1)$ be a solution of $x^2-d(n-1)y^2=1$ with $x_1+1$ divisible by $n-1$ and $y_1/x_1>0$ minimal. Then $$\Mov(X^{[n]}) = \langle H^{[n]}, H^{[n]}-\frac{dy_1}{x_1}B\rangle.$$
    \end{itemize}
\end{corollary}
\begin{remark}
    Solutions to this can of course be obtained by hand. However, the program \cite{python} can compute the wall-and-chamber decomposition of any smooth moduli space of coherent sheaves on K3 surfaces of Picard rank $1$.
\end{remark}
\subsection*{Numerical Observations for K3 surfaces}
For $X^{[3]}$  we can compute the following divisors inducing the walls:
\renewcommand{\arraystretch}{2}
    \begin{longtable}{c|c|c}
        $H.H$ & Wall divisors & Movable cone\\\hline
        $2$ & & $H^{[3]}-\frac{1}{2}B$\\
        $4$ & $H^{[3]}-\frac{4}{5}B$& $H^{[3]}-1B$\\
        $6$ & $H^{[3]}-B$& $H^{[3]}-\frac{6}{5}B$\\
        $8$& & $H^{[3]}-\frac{4}{3}B$\\
        $10$ & $H^{[3]}-\frac{20}{13}B$, $H^{[3]}-\frac{10}{7}B$& $H^{[3]}-\frac{30}{19}B$\\
        $12$ & & $H^{[3]}-\frac{12}{7}B$\\
        $14$ & & $H^{[3]}-\frac{7}{4}B$\\
        $16$ & &$H^{[3]}-2B$ \\       
        $24$ & & $H^{[3]}-\frac{12}{5}B$\\
        $a^2>4$ & & $H^{[3]}- \frac{a}{2}B$\\
        \caption{Computation of the movable cone and its chambers for $X^{[3]}$. The case $a^2> 4$ is proven in \Cref{lem:perfectSquare} below.}
    \label{tab:MovableConesAndWalls}
    \end{longtable}
\begin{lemma}
\label{lem:perfectSquare}
    Let $X$ be a surface with $\Pic(X) = \ZZ H$ and $H.H = a^2 > 4$ a perfect square. Then $$\Mov(X^{[3]}) = \Nef(X^{[3]}) = \langle H^{[3]} ,H^{[3]}-\frac{a}{2}B \rangle.$$
    In particular, there do not exist any flops.
\end{lemma}
\begin{proof}
    The movable cone has this form by \Cref{cor:baymacri}. Thus, we only need to show that no flops exist. By \cite[Theorem 5.7]{bayer2014mmp} this could only happen in two cases: Let $v = (1,0,-2)$.
    \begin{itemize}
        \item There exists an $s\in\Halg(X)$ with $s^2 = -2$ and $s.v \in \{1,2\}$.
        \item There exist $a,b\in \Halg(X)$ that are positive, i.e. $a^2, b^2 \ge 0$, $a.v > 0$ and $b.v > 0$, such that $a+b = v$.
    \end{itemize}
    In this first case with $s.v = 1$ we have that $s^2 = -2$ implies $D^2 = 4r(r-1)-2$. But this is clearly not divisible by $4$, a contradiction to $D^2$ being a perfect square that is even.\par
    If $s.v = 2$ a computation shows that $s^2=-2$ implies $(2r-2)(2r+1)= D^2.$ Again, this contradicts $D^2 > 4$ being an even perfect square.\par
    In the second case, we have that $a.b > 0$ and thus $$4 = v^2 = (a+b)^2 = a^2+2a.b+b^2.$$
    This would mean that without loss of generality $a^2 = 0$ and $a.v \in \{1,2\}.$ However, with $a = (r,D,c)$ this would force $D^2 \in \{2r(2r-1),2r(2r-2)\}$, which is impossible if $D^2>4$ is an even perfect square.
\end{proof}
\noindent We will come to the main observation \ref{obs:mainobservation}, which numerically suggests a connection with the Seshadri constants
\begin{observation}
\label{rem:seshadriequalwall}
     From the calculations above we get that in any case in which the Seshadri number $\varepsilon$ of a very general K3 surface is known, the divisor $H^{[3]}-\frac{\varepsilon}{2}B$ lies on a wall of the movable cone of Hilbert cube $X^{[3]}$. Furthermore, the suggestions for the Seshadri constant of Galati--Knutsen \cite{knutsenSeshadri6and8} in degrees $14,24$ line up with these observations as well.\par
\end{observation}
\begin{remark}[A possible explanation]
In the cases above, the Seshadri constant is computed by rational curves and elliptic curves. We put two remarks on their gonality:
    \begin{itemize}
        \item Let $C$ be a rational curve with two or three nodes. Then there exists a morphism $C\to \PP^1$ of degree $3$.
        \item Let $C$ be an elliptic curve with nodes and triple points. Choose one triple point $x\in C$. Any such curve (-singularity) is trigonal in the following sense: There exists a partial desingularization $\Tilde{C}\to C$ which is an isomorphism around the singularity $x\in C$ together with a morphism $\Tilde{C}\to \PP^1$ of degree $3$.
    \end{itemize}
    This morphism thus induces $\Tilde{C}\subset X\times \PP^1$ and thus a morphism $f_C\colon \PP^1\to X^{[3]}$. This satisfies $f_C(\PP^1).H^{[r]} = C.H$ and $f_C(\PP^1).B = 2+p_a(\Tilde{C})$.
    Thus, for K3 surfaces we get the following curves computing the Seshadri constant:
    \begin{itemize}
        \item Degree $2$: $C\in |H|$ is a rational curve with two nodes, and
        \begin{equation*}
            (H^{[3]}-\frac{1}{2}B).f_C(\PP^1) = 0.
        \end{equation*}
        \item Degree $4$: $C\in |H|$ is a rational curve with three nodes, and
        \begin{equation*}
            (H^{[3]}-\frac{4}{5}B).f_C(\PP^1) = 0.
        \end{equation*}
        \item Degree $6$: $C\in |H|$ is an elliptic curve with a triple point, and
        \begin{equation*}
            (H^{[3]}-B).f_C(\PP^1) = 0.
        \end{equation*}
        \item Degree $8$: $C\in |H|$ is an elliptic curve with a triple point and a node. Then,
        \begin{equation*}
            (H^{[3]}-\frac{4}{3}B).f_C(\PP^1) = 0.
        \end{equation*}
    \end{itemize}
\end{remark}

Focusing on the Farnik--Szemberg--Szpond--Tutaj-Gasinska \Cref{conj:szemberg} it turns out, these numbers also (partially) show up in the movable cone of $X^{[3]}$:
\begin{corollary}
    Let $X$ be a K3 surface with Picard rank $\Pic X = \ZZ H$. Then the movable cone of $X^{[3]}$ has the following form $H^{[3]}-\frac{\alpha_m}{2} B$
    \begin{itemize}
        \item If $H^2$ is a perfect square $\alpha_m = \sqrt{H^2}$
        \item If $2p^2-\frac{H^2}{2}q^2 = 1$ has a solution, the solution $(p,q)$ with $p$ minimal satisfies
        \begin{equation*}
            \alpha_m = \frac{q}{p}H^2.
        \end{equation*}
        \item Otherwise, let $(p,q)$ be a solution of $p^2-H^2q^2 = 1$ such that $p+1$ is divisible by $2$ and with $q/p$ minimal, then
        \begin{equation*}
            \alpha_m = \frac{q}{p} H^2.
        \end{equation*}
    \end{itemize}
\end{corollary}
\begin{proof}
    This follows directly from \Cref{cor:baymacri}.
\end{proof}

\subsection*{Numerical observations for Enriques surfaces}
For Enriques surfaces $X$ that are unnodal, i.e.\@ contain no curves of negative self-intersection, Nuer \cite{nuerEnriques} showed that for any ample $H\in \Pic X$ we have that
\begin{equation*}
    H^{[n]} - \frac{\phi(H)}{n} B
\end{equation*}
is nef but not ample as a divisor on $X^{[n]}$. Thus for $X$ an unnodal very general Enriques surface we have that
    \begin{equation*}
        H^{[n]} - \frac{\varepsilon(H)}{n} B
    \end{equation*}
    is on the boundary of $\Nef X^{[n]}$, lining up with the observations for K3 surfaces.

\section{Background on nested Hilbert schemes}
We start with a recollection of well known facts on nested Hilbert schemes. We follow Lehn's lecture notes \cite{lehnLecturesHilbert} and Ellingsrud--Strømme \cite{ellingsrudNested}.
Let $X$ be a surface and $X^{[r]}$ be the Hilbert scheme of $r$ points. Then by construction there is the ideal sheaf $I_r$ of the universal family on $X^{[r]}\times X$ such that the blowup $\mathrm{Bl}_{I_r} (X^{[r]}\times X)$ parameterizes subschemes of length $r+1$, i.e.\@ there is a morphism
\begin{equation*}
    \mathrm{Bl}_{I_r} (X^{[r]}\times X) \to X^{[r+1]},
\end{equation*}
see e.g.\@ \cite[Section 2]{lehnLecturesHilbert}. Outside the exceptional divisor it is given by sending $(x,y)\in X^{[r]}\times X$ to $x\cup y$.
The induced morphism 
\begin{equation*}
    \mathrm{Bl}_{I_r} (X^{[r]}\times X) \to X^{[r]}\times X^{[r+1]}
\end{equation*}
is a closed immersion and the image is
\begin{equation*}
    X^{[r,r+1]} = \{(z',z)\in X^{[r]}\times X^{[r+1]}\,|\, z'\subset z\},
\end{equation*}
see \cite[Prop. 2.2.]{ellingsrudNested}.
This space admits a residue morphism $\mathrm{res}\colon X^{[r,r+1]}\to X$ sending $(z',z) \mapsto z\setminus z'$. Then the following diagram
\begin{center}
\begin{tikzcd}
                                      & {X^{[r]}}   &                                                        \\
{X^{[r,r+1]}} \arrow[ru, "p_{r}"] \arrow[rd, "p_{r+1}"] \arrow[rdd, "\mathrm{res}"'] \arrow[rr, "\cong"]&             & {\mathrm{Bl}\,(X^{[r]}\times X)} \arrow[lu] \arrow[ld]\arrow[ldd] \\
                                      & {X^{[r+1]}} & \\
                                      & X &
\end{tikzcd}    
\end{center}
commutes. 

\begin{notation}
    We denote by $E\in \Pic(X^{[r,r+1]})$ the exceptional divisor coming from the blowup description and for any $H\in \Pic X$ we denote $\textup{res}^* H = H\diff$.\par
    If no confusion is possible, we further identify any divisor in $\Pic (X^{[r]})$ and $\Pic (X^{[r+1]})$ with its pullback in $\Pic X^{[r,r+1]}$, e.g.\@
    if $H\in \Pic X$ we denote the divisors $p^*_rH^{[r]}$ and $p^*_{r+1}H^{[r+1]}$ by $H^{[r]}$ and $H^{[r+1]}$ respectively. We proceed analogously with the divisors $B^{[r]}, B^{[r+1]}$. 
\end{notation}
In \cite[Section 3.2 and 4.2]{ryanNefConeNested} it is shown\footnote{Note, that in mentioned paper the notation is slightly different: The divisor $B^{[r]}$ differs from their notion by a factor of $2$. We use the notation as in \cite[Section 13]{bayer2014mmp}.} that in the Picard group we have
    \begin{align*}
        H^{[r+1]} &= H\diff + H^{[r]}\\
        B^{[r+1]} &= E + B^{[r]}.
    \end{align*}
The fibers of $X^{[r,r+1]}\to X^{[r]}\times X$ are as follows, see e.g.\@ \cite{ellingsrudNested, ryanNefConeNested}. For any given $(\xi, P)\in X^{[r]}\times X$, the preimage is the projective space $\PP(I_\xi(P))$. \par
For any zero-dimensional closed subscheme $\xi\subset X$ we want to find an embedding of the blowup 
$p\colon \Bl_\xi X\to X$ into $X^{[r,r+1]}$. In the following the \emph{exceptional divisor} denotes the closed subscheme coming from the universal property of blowing up, i.e. the closed subscheme corresponding to the ideal sheaf $p^{-1} I_\xi \cdot \oO_{\Bl_\xi X}$, which is a line bundle.

\begin{lemma}
\label{lem:ExceptionalDivOnNested}
    Let $\xi \in X^{[r]}$ be a point. Then there is a natural inclusion of $\Bl_\xi X \hookrightarrow X^{[r,r+1]}$ contained in the fiber of $X^{[r,r+1]}\to X^{[r]}$ over $\xi$ and the composition $\Bl_\xi X \hookrightarrow X^{[r,r+1]}\to X$ is precisely the blowup along $\xi$. Moreover the exceptional divisor of the morphism $X^{[r,r+1]}\to X^{[r]}\times X$ restricts to the exceptional divisor of the blowup $\Bl_\xi X\to X$.
\end{lemma}
\begin{proof}
    By the discussion above $X^{[r,r+1]} = \Bl_{I_Z} (X^{[r]}\times X)$, where $Z$ is the universal family and $I_Z$ the ideal sheaf. Let $p\colon X^{[r]}\times X\to X^{[r]}$ be the projection and $X_\xi = p^{-1}(\xi)\xrightarrow{i} X^{[r]}\times X$ the fiber over $\xi$. By definition of the universal family $i^{-1}I_Z\cdot \oO_{X_\xi} = I_\xi$ holds and the universal property of blowing up yields the closed immersion $\Bl_{\xi} X\to X^{[r,r+1]}$ with the desired properties, see e.g.\@ \cite[Prop. II.7.14 and Cor. II.7.15]{HartshorneBook}.
\end{proof}
\begin{remark}
\label{rem:fiber_of_nested}
To conclude, the fiber of $X^{[r,r+1]}\to X^{[r]}$ over a point $\xi \in X^{[r]}$ is the union 
\begin{equation*}
    \Bl_\xi X \cup \PP^{n_0} \cup\ldots \cup \PP^{n_i}
\end{equation*}
and the projective spaces meet $\Bl_\xi X$ in the exceptional divisor. Moreover for $\xi$ a union of disjoint points, the fiber is just the usual blowup along $\xi$.
\end{remark}

\section{Cones of Nested Hilbert schemes}
\label{sec:ConesOfNested}
In this section we will show that the Seshadri constants are computed by either the nef cone or the effective cone of the nested Hilbert schemes. \par
For this section let $X$ be a surface.
We start by linking ampleness on the blow ups of the underlying surface with ampleness on the nested Hilbert scheme: 
\begin{lemma}
\label{lem:fundamental}
    Let $X$ be a surface, $p_r\colon X^{[r,r+1]} \to X^{[r]}$ be the canonical projection and let $\lambda > 0$. Then for any $Z\in X^{[r]}$ the divisor $H\diff - \lambda E$ restricted to the fiber $p_r^{-1}(Z)$ is ample if and only if $H-\lambda E_Z$ is ample on the blow up $\Bl_Z X$, where $E_Z$ is the exceptional divisor.
\end{lemma}
\begin{proof}
    Let $p_r\colon X^{[r,r+1]} \to X^{[r]}$ be the canonical projection and let $\lambda > 0$. For any $Z\in X^{[r]}$ the fiber $p_r^{-1}(Z)$ is the union of some projective spaces $\PP^{n_i}$ and the blowup $\Bl_Z X$ by \Cref{rem:fiber_of_nested}. By 
    \Cref{lem:ExceptionalDivOnNested} the restriction satisfies
    \begin{equation*}
        (H\diff - \lambda E)|_{\Bl_Z X} = H-\lambda E_Z,
    \end{equation*}
    where $E_Z$ is the exceptional divisor of the blowup. Moreover, the restriction to the projective spaces $(H\diff - \lambda E)|_{\PP^{n_i}} = -\lambda E|_{\PP^{n_i}}$ is ample as $\lambda >0$. Thus, by \cite[Prop 1.2.16]{lazarsfeldPositivity} any line bundle as in the theorem is ample on the fiber if and only if it is ample on $\Bl_Z X$.
\end{proof}

\begin{proposition}
\label{prop:supSeshadri}
    For a surface $X$ and an ample line bundle $H$ we have
    \begin{equation*}
        \varepsilon(H,r) = \varepsilon_{\textup{sup}}(H, r).
    \end{equation*}
\end{proposition}
\begin{proof}
    By the previous \Cref{lem:fundamental} for any $\lambda < \varepsilon_{\textup{sup}}(H, r)$ the line bundle $H\diff -\lambda E$ is ample on at least one fiber. Therefore openness of ampleness in families \cite[Prop 1.2.17]{lazarsfeldPositivity} shows that it is ample for the  general $Z\in X^{[r]}$. 
\end{proof}
\subsection*{The nef cone of the nested Hilbert scheme}
The nef cone of the nested Hilbert scheme is directly be linked to the infimum Seshadri constant: Intuitively, the fibers of the morphism $X^{[r,r+1]}\to X^{[r]}$ parameterizes the surface blown up at $r$ points and therefore, ampleness on the nested Hilbert scheme is linked to the ampleness of the fibers as follows:
\begin{theorem}
\label{thm:vareps}
    Let $X$ be a surface and $X^{[r,r+1]}$ the nested Hilbert scheme for some $r\in \NN_{\ge 1}$. Then 
    \begin{equation*}
        \sup_{\lambda}\{\lambda>0 \,|\, H\diff-\lambda E + A\textup{ is ample for some ample }A\in \Pic(X^{[r]})\} = \varepsilon_{\textup{inf}}(H, r).
    \end{equation*}    
\end{theorem}
\begin{proof}
    Let $p_r\colon X^{[r,r+1]} \to X^{[r]}$ be the canonical projection. By \Cref{lem:fundamental} if $\lambda < \varepsilon_{\textup{inf}}(H,r)$ the divisor $H\diff-\lambda E$ is $p_r$-ample. By \cite[Prop 1.7.10]{lazarsfeldPositivity} tensoring with a positive enough ample line bundle $A\in \Pic X^{[r]}$ we get that $H\diff-\lambda E + A$ is ample.
    On the other hand, let $\lambda > \varepsilon_{\textup{inf}}(H, r)$. Then by the Nakai-Moishezon criterion there exists a $Z\in X^{[r]}$ and a curve $C\in \mathrm{Bl}_Z\, X$ such that $(H\diff - \lambda E)|_{\Bl_Z X}.C < 0$ as $E$ restricts to the exceptional divisor on the blowup. Thus, now regarding the curve $C$ as a curve in $X^{[r,r+1]}$ via the inclusion $\Bl_Z X \subset X^{[r,r+1]}$ we get
    \begin{equation*}
        \left(H\diff - \lambda E + A\right).C < 0
    \end{equation*}
    for any divisor $A\in \Pic X^{[r]}$, as $A.C = 0$ as $C$ lies in a fiber of $p_r\colon X^{[r,r+1]}\to X^{[r]}$.
\end{proof}

\subsection*{The effective cone of the nested Hilbert scheme}
Let $X$ be a surface and $H$ be an ample divisor.
In this section we will link the effective cone of nested Hilbert schemes to the supremum Seshadri constant $\varepsilon_{\textup{sup}}(H, r)$. To do so, we recall the following lemma from \cite{szembergremarks}: A curve $C$ is Nagata submaximal for some $r$ with respect to $H$ if 
\begin{equation*}
    \frac{C.H}{\sum_{i=1}^r m_i} < \sqrt{\frac{H.H}{r}}
\end{equation*}
for $r$ points with multiplicity $m_i$.
\begin{lemma}
\label{cor:symmetricAction}
    Let $X$ be a surface with $\Pic X = \ZZ H$. Then, the divisor class of the strict transform on the blowup of $r$ very general points of a submaximal Nagata curve is of the form 
    \begin{equation*}
        C = D - m(E_1+\ldots+ E_r) - kE_i
    \end{equation*}
    for some fixed $D\in \Pic X$ and any $C_\sigma = D - m(E_{\sigma(1)}+\ldots+ E_{\sigma(r)}) - kE_{\sigma(i)}$ is effective for any permutation $\sigma\in S_r$.
\end{lemma}
\begin{proof}
    See \cite[Cor 2.8 and its proof]{szembergremarks}.
\end{proof}
\begin{remark}
\label{rem:EdgeEffConeBlowUp}
    We remark the following implications of the corollary above, as already noted by  Dumnicki--Küronya--Maclean--Szemberg 
 in \cite[Lemma 2.2]{szembergRationality}: If $X$ is a surface with $\Pic X = \ZZ H$ and $\Bl_r X \to X$ is the blowup of $r$ very general points with exceptional divisors $E_i$ then 
    \begin{equation*}
        H-\frac{H^2}{r\varepsilon(H,r)}\sum_{i=1}^r E_i\in \partial\Eff (\Bl_r X)
    \end{equation*}
    is $\QQ$-effective and contained in the boundary of the pseudo-effective cone. \par
    If there are no submaximal Nagata curves the statement is clear, thus we assume that there is a submaximal curve $C_i = D - m(E_1+\ldots+ E_r) - kE_i$ for some $1\le i\le r$, $m,k\in \ZZ$ and $D\in\Pic X$ that computes the Seshadri constant. By the corollary above also $C_j$ is effective for all $1\le j\le r$ and therefore, the divisor
    \begin{equation*}
        \sum_{i=1}^r C_i = rD -\left(\sum_{i=1}^r m_i\right)\sum_{i=1}^r E_i
    \end{equation*}
    is effective. As $H-\varepsilon(H,r)\sum_{i=1}^r E_i$ is nef and $(H-\varepsilon(H,r)\sum_{i=1}^r E_i).\sum_{i=1}^r C_i = 0$ this implies that $\sum C_i$ is not big and thus, the claim follows.
\end{remark}
With this we can construct a divisor that parameterizes the submaximal Nagata curves: Recall that $E\subset X^{[r,r+1]}$ is the exceptional divisor of the blow up $X^{[r,r+1]} \to X^{[r]}\times X$.

\begin{theorem}
\label{thm:SeshadriNagataOnNested}
    Let $X$ be a surface with $\Pic X = \ZZ H$. We have
    \begin{equation*}
        \frac{H^2}{r\varepsilon(H,r)} = \sup_\lambda \{ \lambda \,|\, H\diff- \lambda E + A \textup{ is }\QQ\textup{-effective for some ample } A\in\Pic X^{[r]}\}.
    \end{equation*}
\end{theorem}
\begin{proof}
    Let $\Bl_r X \to X$ be the blow up at $r$ very general points. 
    At first, we will show that if $H-\lambda\sum_{i=1}^r E_i$ is effective for some $\lambda > 0$, then also $H\diff - \lambda E + A$ is effective for some ample $A\in \Pic X^{[r]}$. This together with \Cref{rem:EdgeEffConeBlowUp} shows the inequality LHS $\le$ RHS in the theorem.\par
    For $p\colon X^{[r,r+1]}\to X^{[r]}$ we see that the divisor $D = H\diff-\lambda E\in \Pic X^{[r,r+1]}$ restricts to the divisor $H-\lambda \sum_{i=1}^r E_i$ on the fibers of $p$ which are isomorphic to the blowup of $r$ distinct points. Thus, by assumption $D$ restricted to the very general fiber is effective. \par
    Therefore, $p_*\oO(D)$ is non-zero and by tensoring with some ample $A$ we get
    \begin{equation*}
        H^0(X^{[r,r+1]}, D\otimes p^*A) = H^0(X^{[r]}, p_*\oO(D)\otimes A) \neq 0,
    \end{equation*}
    which gives the claim.\par
    To show the other inequality, we again observe that if $H\diff- \lambda E + A$ is effective for some $A\in \Pic X^{[r]}$, then the restriction to the very general fiber of $p$ is again effective: Otherwise the divisor would either contain the very general fiber or not meet it at all: But these two cases cannot happen, as $D$ restricted to a fiber is non-trivial. Thus, 
    \begin{equation*}
        D|_{\Bl_r X} = H-\lambda \sum E_i
    \end{equation*}
    is effective for the very general blow up. That is, $H-\frac{H^2}{r\lambda} \sum E_i$ is not ample and we get that $\varepsilon(H, r)\le \frac{H^2}{r\lambda}$. In other words, 
    \begin{equation*}
        \frac{H^2}{r\varepsilon(H,r)} \ge \lambda.
    \end{equation*}
\end{proof}

\section{Examples of infimum Seshadri constants and a question of Ryan--Yang}
\label{sec:applications}
\subsection*{An example for rational surfaces and K3 surfaces}
Contrary to the case of Hilbert schemes the cones of nested Hilbert schemes are not well understood. However, Ryan--Yang \cite{ryanNefConeNested} computed the nef cone of $X^{[r,r+1]}$ for some choices of $X$ and $r$. Thereby we can calculate the infimum Seshadri constants in these cases.

\begin{example}
    Let $X = \PP^2$ be the projective plane and $H$ a hyperplane. In \cite[Prop. 5.2]{ryanNefConeNested} it was shown that for $r\ge 2$ the nef cone of the nested Hilbert scheme for $\PP^2$ is spanned by 
    \begin{equation*}
    \label{eq:generatingsetp2}
        H\diff, H^{[r]}, H^{[r]}-\frac{1}{r-1}B^{[r]}, H^{[r+1]}-\frac{1}{r}B^{[r+1]}.
    \end{equation*}
    This shows that $\varepsilon_{\textup{inf}}(H, r) = \frac{1}{r} = \frac{1}{r}\varepsilon(H)$.
\end{example}
\begin{example}
    For $e>0$ let $\FF_e$ be the Hirzebruch surface $\PP(\oO_{\PP^1}\oplus\oO_{\PP^1}(-e))$. Denote by $C$ the unique section with $C.C=-e$ and let $F$ be a fiber.
    In \cite[Prop. 5.3]{ryanNefConeNested} it was shown that the nef cone of the nested Hilbert scheme of $\FF_e$ is generated by
    \begin{align*}
        F\diff, C\diff+eF\diff, C^{[r]}+eF^{[r]}, F^{[r]}, C^{[r]}+(e+1)F^{[r]}-\frac{1}{r}B^{[r]},\\ C^{[r+1]}+(e+1)F^{[r+1]}-\frac{1}{r+1}B^{[r+1]},
    \end{align*}
    when $r\ge 2$, see also Bertram--Coskun \cite[Theorem 1]{bertramBirGeometryHirzebruch} for the nef cone of $\FF_e^{[r]}$.
    For an ample $H = aC+bF$ this yields
    \begin{equation*}
        \varepsilon_{\textup{inf}}(H, r) =\frac{1}{r} \min (a, b-ae) = \frac{1}{r} \varepsilon(H).
    \end{equation*}
\end{example}
\begin{remark}
    The behaviour experienced in the examples above is expected in the following sense: If $\varepsilon(H)$ is attained by a smooth curve, we get for the infimum over distinct points $$\inf_{x_1,\ldots, x_r\in X}{\varepsilon}(H, x_1, \ldots, x_r) = \frac{1}{r}\varepsilon(H),$$
    as a simple calculation shows that $\varepsilon(H, x_1,\ldots, x_r)\ge \frac{1}{r}\varepsilon(H)$ for any collection of distinct points $x_i \in X$ and any surface.
\end{remark}
\begin{example}
    Let $X$ be a very general K3 surface with $\Pic X = \ZZ H$. 
    By Ryan--Yang \cite{ryanNefConeNested} the nef cone of $X^{[r,r+1]}$ is spanned by 
    \begin{equation*}
        H\diff, H^{[r]}, H^{[r]}-\frac{H^2}{r+\frac{H^2}{2}}B^{[r]}, H^{[r+1]}-\frac{H^2}{r+1+\frac{H^2}{2}}B^{[r+1]}, 
    \end{equation*}
    when $r\ge \frac{H^2}{2}+1$. Thus, for these choices of $r$ we have
    \begin{equation*}
        \varepsilon_{\textup{inf}}(H,r) = \frac{H^2}{r+1+\frac{H^2}{2}}.
    \end{equation*}
    In fact, the Seshadri constant is computed by a nodal rational curve $R$ in the primitive linear system $|H|$: These curves possess $\frac{H^2}{2}+1$ nodes, i.e. choosing these nodes and $\frac{H^2}{2}+1-r$ remaining smooth points gives the claim.
\end{example}
It turns out that for $r\gg 0$ this is the general behaviour, i.e.\@ the class of the curve computing the Seshadri constant is minimal: We need to introduce the slightly altered Seshadri constant by allowing only distinct points:
\begin{equation*}
    \Tilde{\varepsilon}_{\textup{inf}} (H,r) = \inf_{x_1,\ldots, x_r}\varepsilon(H, x_1,\ldots, x_r)
\end{equation*}
\begin{theorem}
    Let $X$ be a surface and $H$ an ample divisor. Then there is an $r_0\in \NN$ and a curve $C\subset X$ of minimal degree $C.H = \min \{C'.H\,|\, C'\subset X \textup{ a curve}\}$ such that 
    \begin{equation*}
        \Tilde{\varepsilon}_{\textup{inf}}(H,r) = \frac{C.H}{\sum_{i=1}^{r_0} m_i+r-r_0}
    \end{equation*}
    for all $r\gg 0$, where $m_i\in \NN$ are the multiplicities of $C$ at $r_0$ distinct points $x_i\in C$.
\end{theorem}
\begin{proof}
    Let $K_X$ be the canonical divisor and let $D\subset X$ be a curve such that $D.H$ is minimal, i.e.\@ $D.H = \min \{C'.H\,|\, C'\subset X \textup{ a curve}\}$. Suppose there is an irreducible curve $C\subset X$ with multiplicities $m_i$ at $r$ distinct points $x_i\in C$ such that
    \begin{equation*}
        \frac{C.H}{\sum_{i=1}^r m_i} < \frac{D.H}{r}.
    \end{equation*}
    We want to show that $C.H = D.H$ holds if $r\gg 0$. For that we observe that by assumption
    \begin{equation*}
        r\frac{C.H}{D.H}< \sum_{i=1}^r m_i.
    \end{equation*}
    Denote by $\Bl_r X\to X$ the blow up of $X$ in the points $x_i\in C$ with $E_i$ the exceptional divisors and let $\Tilde{C}\subset \Bl_r X$ be the strict transform of $C$. 
    Then, $K_{\Bl_r X} = K_X + \sum_{i=1}^r E_i$ and $\Tilde{C}= C-\sum_{i=1}^r m_i E_i$ by \cite[Chapter V.3]{HartshorneBook}.
    Hence, by the adjunction formula the arithmetic genus $p_a$ of $\Tilde{C}$ satisfies
    \begin{align*}
        2p_a(\Tilde{C})-2 &= (K_{\Bl_r X}+\Tilde{C}).\Tilde{C}\\
        &= C^2+K_X.C - \sum_{i=1}^r m_i^2 + \sum_{i=1}^r m_i\\
        &\le C^2+K_X.C -\frac{\left(\sum_{i=1}^r m_i\right)^2}{r} +\sum_{i=1}^r m_i\\
        &\le C^2+K_X.C - \frac{C.H}{D.H}\sum_{i=1}^r m_i +\sum_{i=1}^r m_i\\
        &\le  C^2+K_X.C - r(\frac{C.H}{D.H}-1)\frac{C.H}{D.H}.
    \end{align*}
    We will show that the latter is $<-2$ if $C.H > D.H$ and $r\gg 0$.
    For the $K_X.C$ term we observe that $\max_{\{C \textup{ is effective}\}} \frac{C.K_X}{C.H}<\infty$ as $H$ is ample, i.e.\@ $K_X.C \le c\,C.H$ for a constant $c$ only depending on the surface and $H$.\par
    By the Hodge Index theorem $C^2\le \frac{(C.H)^2}{H^2}$.
    Thus, with $r' = \frac{r}{(D.H)^2}$ we get
    \allowdisplaybreaks[1]
    \begin{align*}
        C^2-r(\frac{C.H}{D.H}-1)\frac{C.H}{D.H} &= C^2-r'(C.H-D.H)C.H\\
        &\le \frac{(C.H)^2}{H^2}-r'(C.H-D.H)C.H\\
        &= C.H(C.H(\frac{1}{H^2}-r')+r'D.H)\\
        &\le C.H((D.H+1)(\frac{1}{H^2}-r') + r'D.H)\\
        &= CH(\frac{D.H}{H^2}+\frac{1}{H^2}-r').
    \end{align*}
    Putting everything together, we get
    \begin{align*}
        2p_a(\Tilde{C})-2 &\le K_X.C + CH(\frac{D.H}{H^2}+\frac{1}{H^2}-r') \\
        &\le C.H(c+\frac{D.H}{H^2}+\frac{1}{H^2}-r') \ll 0
    \end{align*}
    for $r\gg 0$ independent of the curve $C$, a contradiction.\par
    As we observed before, $2p_a(C)-2 = K_X.C + C.C$ is bounded by a constant (depending only on the surface and $H$) times $(C.H)^2$. Thus, there is a maximum $r_0$ of singular points among all curves with $C.H$ minimal. Choosing a curve $C_0$ such that $\sum_{i=1}^{r_0} m_i$
    is maximal, where the $m_i$ are multiplicities at (not necessarily singular) points, this curve $C_0$ indeed computes the Seshadri constant $\Tilde{\varepsilon}_{\textup{inf}}(X,H,r)$ for all $r\gg 0$ with
    \begin{equation*}
        \Tilde{\varepsilon}_{\textup{inf}}(H,r)= \frac{C_0.H}{\sum_{i=1}^{r_0} (m_i-1) + r}.
    \end{equation*}
\end{proof}

\subsection*{An answer to a question of Ryan--Yang}
In the paper \cite{ryanNefConeNested} Ryan--Yang posed the following question after observing that for $r\gg 0$ the nested Hilbert schemes $\PP^{2[r,r+1]}$ are not log-Fano:
\begin{question*}
    For which $r\gg0$ is the nested Hilbert scheme $\PP^{2[r,r+1]}$ a Mori-Dream-Space?
\end{question*}
We will answer this for the more general case of $X$ being a surface with $\Pic X = \ZZ H$ with the assumption of the Nagata--Biran-\Cref{conj:NagataBiran}.
\begin{proposition}
    Let $X$ be a surface with $\Pic X = \ZZ H$ and $r\in \NN$ such that $\sqrt{\frac{H^2}{r}}\notin\QQ$. Suppose
    \begin{equation*}
        \varepsilon(H,r) = \sqrt{\frac{H^2}{r}}.
    \end{equation*} 
    Then $X^{[r,r+1]}$ is not a Mori-Dream-Space.
\end{proposition}
\begin{proof}
    The pseudo-effective cone of a Mori-Dream-Space is rational polyhedral, see \cite[Prop. 1.11]{hikeelMoriDreamSpace}. This contradicts the assumption $\sqrt{\frac{H^2}{r}}\notin \QQ$ together with \Cref{thm:SeshadriNagataOnNested}.
\end{proof}

\begin{proposition}
    The spaces $\PP^{2[r,r+1]}$ are not Mori-Dream-Spaces for all perfect squares $r = s^2>9$.
\end{proposition}
\begin{proof}
    Let $\Bl_r \PP^2$ be the blow up of $r$ very general points with exceptional divisors $E_i$.
    In his original paper \cite{nagataOriginalPaper} Nagata shows, that $H-\sqrt{\frac{1}{r}}\sum_{i=1}^r E_i$ is not $\QQ$-effective and thus $\varepsilon(H, r) = \sqrt{\frac{1}{r}}$. By \Cref{thm:SeshadriNagataOnNested} we have
    \begin{equation*}
        \sup_\lambda \{ \lambda \,|\, H\diff- \lambda E + A \textup{ is }\QQ\textup{-effective for some ample } A\in \Pic X^{[r]}\} = \sqrt{\frac{1}{r}} = \frac{1}{s}.
    \end{equation*}
    
    On the other hand, by \cite[Prop. 1.11]{hikeelMoriDreamSpace} the pseudo-effective cone of a Mori-Dream-Space is generated by finitely many effective divisors. Therefore, there exists an effective divisor $D$ that is a rational multiple of $H\diff-{\frac{1}{s}}E + A$ for some $A\in \Pic X^{[r]}$. However, this divisor would restrict to an effective divisor on the blowup $\Bl_{r} X\subset X^{[r,r+1]}$ at $r$ very general points, a contradiction to Nagata's result.
\end{proof}

\section{Some new bounds on Seshadri constants}
\label{sec:BoundsOnSeshadri}
In this section we will provide new bounds on one point Seshadri constants. The strategy is to embed $\Bl_x X$ into $X^{[r]}$ for some large $r$ and get some bounds via the nef cone of the Hilbert scheme which are computed in \cite{BologneseNefConeHilbert} via Bridgeland stability conditions. For some fixed $n$ denote $N = \binom{n}{2}$.
\begin{proposition}
\label{prop:BlowUpPower}
    Let $X$ be a surface, $x\in X$ a point and $n> 0$ an integer. Then the blowup of the ideal $\mm_x^n$ naturally leads to an embedding $\Bl_x X \subset X^{[N,N+1]}$ such that $E, H\diff\in \Pic X^{[N,N+1]}$ restrict to $nE_x, H$ on $\Bl_x X$ respectively, where $E_x$ is the exceptional divisor.
\end{proposition}
\begin{proof}
    By \Cref{lem:ExceptionalDivOnNested} we can embed $\Bl_{\mm_x^n} X\subset X^{[N, N+1]}$, as $\mm_x^n$ defines a closed subscheme in $X$ with support $x$ of length $N$.
    
    Again, by \Cref{lem:ExceptionalDivOnNested} the exceptional divisor $E$ of the nested Hilbert scheme restricts to the exceptional divisor $E_{\mm^n_x}$ on the blowup $\Bl_{\mm^n_x} X$.
    On the other hand, the natural isomorphism 
    \begin{equation*}
        p\colon \Bl_{\mm_x} X \to \Bl_{\mm_x^n} X,
    \end{equation*}
    satisfies $p^*E_{\mm^n_x} = nE_x$, where $E_x$ is the exceptional divisor of $\Bl_x X$, giving the claim. 

    The statement for $H\diff$ and $H$ directly follows from \Cref{lem:ExceptionalDivOnNested} and the fact $p^*H = H$.
\end{proof}
\begin{proposition}
\label{seshadribound1}
    Let $X$ be a surface with Picard group $\Pic X = \ZZ H$ with effective generator $H$ such that $K_X = aH$ with $a\ge 0$. Then
    \begin{equation*}
        \varepsilon(H) \ge \frac{2H^2n}{(a+1)H^2+n(n+1)+2}
    \end{equation*}
    for all integers $n$ such that $n^2+n \ge H^2$.    
\end{proposition}
\begin{proof}
    By \cite[Prop 4.4.]{BologneseNefConeHilbert} the nef cone of the Hilbert scheme $X^{[r+1]}$ contains  the  cone spanned by
    \begin{equation*}
        H^{[r+1]}, \left(\frac{a+1}{2}+\frac{r+1}{H^2}\right)H^{[r+1]}+ B^{[r+1]}
    \end{equation*}
    for $r+1\ge \frac{H^2}{2}+1$. Therefore, this applies to our case when $r = N$. Pulling back the nef classes via the morphism $X^{[r,r+1]}\to X^{[r+1]}$ \Cref{prop:BlowUpPower} gives that
    \begin{equation*}
        \left(\frac{a+1}{2}+\frac{N+1}{H^2}\right)H+ nE
    \end{equation*}
    is nef on $\Bl_x X$ for any $x\in X$, giving the claim.
\end{proof}
\begin{corollary}
    \label{seshadribound2}
    Let $X$ be a K3 surface with Picard group $\Pic X = \ZZ H$. Then 
    \begin{equation*}
        \varepsilon(X, H) \ge \frac{2H^2\alpha}{H^2+\alpha(\alpha+1)+2}.
    \end{equation*}
    where $\alpha =\left\lceil \sqrt{H^2+2}\right\rceil$.\qedhere
\end{corollary}
\begin{corollary}
\label{seshadribound3}
    Let $X$ be a surface of general type with Picard group $\Pic X = \ZZ H$, $K_X = aH$ and effective generator $H$. Then 
    \begin{equation*}
        \varepsilon(X, H) \ge \frac{2H^2n}{(a+1)H^2+n(n+1)+2}
    \end{equation*}
    for any $n\in \NN$.    
\end{corollary}
\begin{corollary}
    Let $X$ be a surface of general type with Picard group $\Pic X = \ZZ H$ and $p_g\neq 0$. Then 
    \begin{equation*}
        \varepsilon(X, H) \ge \frac{2bH^2\alpha}{(ab+b^2)H^2+\alpha(\alpha+1)+2}
    \end{equation*}
    for $\alpha = \left\lceil\sqrt{(ab+b^2)H^2+2}\right\rceil$, where $K_X = aH$ and $b$ is minimal with $bH$ effective.
\end{corollary}
\begin{proof}
    This follows as before, but with \cite[Theorem 4.2]{BologneseNefConeHilbert} which takes the case into account that $H$ might not be effective. 
\end{proof}

In the following we give a table of explicit bounds and compare these to known results. 
Let $X$ be a K3 surface with $\Pic X = \ZZ H$. In \cite{knutsenGeneralK3}, Knutsen gave the bound $$\varepsilon(X, H)\ge \lfloor{\sqrt{H^2}}\rfloor,$$ except for some cases in which the bound is $\lfloor\sqrt{H^2}\rfloor-\frac{1}{2\lfloor\sqrt{H^2}\rfloor+1}$. We compare that with our result, which we call  $\alpha_{H.H}$ for simplicity. It turns out, that our bounds are better as soon as $\sqrt{H^2+2}$ is not far from an integer. Numerically testing the degrees up to $H.H = 10^6$, we find that our bound is better in approximately $50$ percent of the cases. We show some cases in which the new bound is better.

    \begin{longtable}{c|c|c|c}
$H.H$& $\alpha_{H.H}$&  $\lfloor\sqrt{H^2}\rfloor$ &  $\sqrt{H^2}$\\\hline
$8$ & $2.133$ & $2$ & $2.828$ \\
$14$ & $3.111$ & $3$ & $3.742$ \\
$22$ & $4.074$ & $4$ & $4.690$  \\
$24$ & $4.235$ & $4$ & $4.899$\\
$32$ & $5.053$ & $5$ & $5.657$  \\
$34$ & $5.231$ & $5$ & $5.831$  \\
$58$ & $7.030$ & $7$ & $7.616$\\
$60$ & $7.164$ & $7$ & $7.746$\\
$62$ & $7.294$ & $7$ & $7.874$\\
$74$ & $8.024$ & $8$ & $8.602$\\
$76$ & $8.143$ & $8$ & $8.718$\\
$78$ & $8.259$ & $8$ & $8.832$\\
$80$ & $8.333$ & $8$ & $8.944$\\

\caption{Bounds for the Seshadri numbers}
\label{tab:SeshadriNumbersK3Bounds}
    \end{longtable}

\section{Questions}
As we have seen in the previous chapter, embedding the blowup of a surface into some Hilbert schemes makes it possible to tackle problems regarding Seshadri constants. We will give some questions regarding this topic:
\begin{question*}
    What is the relation between the Seshadri constant $\varepsilon(X,H)$ of an ample line bundle $H$ on a surface $X$ and the movable and nef cone of the Hilbert scheme $X^{[3]}$?
\end{question*}
\begin{remark}
    Let $X$ be a K3 surface with Picard group $\Pic X = \ZZ H$. If a curve $C\subset X$ computes the Seshadri constant at a point $x\in X$, we have the following exact sequence on the blowup at the point $x$
    \begin{equation*}
        0\to \oO(-\Tilde{C}-2E) \to \oO(-2E) \to \oO_{\Tilde{C}}(-2E)\to 0,
    \end{equation*}
    where $\Tilde{C}$ is the strict transform.
    This induces an exact triangle
    \begin{equation*}
        \mathbb{R}p_*\oO(-\Tilde{C}-2E) \to \mm^2 \to p_*\oO_{\Tilde{C}}(-2E)
    \end{equation*}
    in $D^b(X)$. Is this a destabilising sequence of $\mm^2$ for some Bridgeland stability condition on $D^b(X)$? If so, the corresponding wall, would induce a movable divisor with class $H^{[3]}-\frac{pH^2}{2m+1-p_a(\Tilde{C})}B$, lining up with some of our observations.
\end{remark}

On the other hand, we have seen that there is also a numerical connection between a conjectured bound of $\varepsilon(H,1)$ and $X^{[3]}$ in \Cref{conj:szemberg}. This poses the question if it is possible to prove the latter with Bridgeland stability methods or at least answer the following weaker questions:
\begin{question*}
    Suppose $X$ is a K3 surface with $\Pic X = \ZZ H$ and $H^2$ not a perfect square. Is
    \begin{equation*}
        \varepsilon(H,1) \ge \frac{q}{p}H^2
    \end{equation*}
    for at least one solution of $2p^2-\frac{H^2}{2}q^2 = 1$ or $p^2-H^2q^2 = 1$?
\end{question*}

\begin{question*}
    Is $\varepsilon_{\textup{inf}}(X,H) = \Tilde{\varepsilon}_{\textup{inf}}(X,H)$ for all $r\gg 0$ or even $r>0$?
\end{question*}

It might also be interesting to generalise the technique of \Cref{prop:BlowUpPower} to several points to gain an understanding of $\varepsilon(H,r)$. That is, for any integral zero-dimensional subscheme $Z\in X^{[r]}$ with ideal sheaf $I_Z$ consider the blowup 
\begin{equation*}
    \Bl_{I_Z} X\cong \Bl_{I_Z^n} X \subset X^{[r\frac{n(n+1)}{2}, r\frac{n(n+1)}{2}+1]}.
\end{equation*} 
Restricting the nef cone as in \Cref{prop:BlowUpPower} gives us new bounds for $\varepsilon_{\textup{sup}}(H,r)$.
However, as expected this does not yield good results for $\varepsilon(H,r)$ for $r\gg 0$. We therefore turn our attention to another invariant of those Hilbert schemes.\par
Let $X$ be a K3 surface with Picard rank $1$. Then restricting the non-trivial ray of the movable cone of $X^{[r\frac{n(n+1)}{2}]}$ to $\Bl_{I_Z^n} X \cong \Bl_{I_Z} X$ yields a divisor $H-\alpha_n \sum E_i$ such that $\alpha_n\xrightarrow{n\to\infty} \sqrt{\frac{H^2}{r}}$. With regards to the Nagata--Biran conjecture this poses the question:
\begin{question*}
    For which $n>0$ does the blowup $\Bl_{I_Z^n} X$ of the very general $Z$ not meet the stable base locus of any effective divisor $D\in \Pic X^{[r\frac{n(n+1)}{2}+1]}$ in a positive dimensional subset?
\end{question*}

\printbibliography
\end{document}